\documentclass[a4paper, 10pt]{article}
\usepackage{graphicx}
\usepackage{alltt,placeins,caption}
\usepackage{amsmath,amsthm,amssymb}
\usepackage[hidelinks, pdftex]{hyperref}

\newtheorem{lemma}{Lemma}
\newtheorem*{theorem*}{Main Result}
\usepackage[backend=biber]{biblatex}
\begin{document}

\begin{center}
\LARGE
\textbf{On the Number of Disconnected Character Degree Graphs Satisfying Pálfy's Inequality}\\[20pt]
\small
\textbf {Mark L. Lewis, Andrew Summers}\\[20pt]

Department of Mathematical Sciences, Kent State University, Kent, OH 44242, USA \\lewis@math.kent.edu\\[10pt]

Department of Mathematical Sciences, Kent State University, Kent, OH 44242, USA \\ asumme19@kent.edu\\[20pt]
%Received: date\quad/\quad
%Revised: date\quad/\quad
%Publised online: data
\end{center}

\begin{abstract}
Let $G$ be a finite solvable group with disconnected character degree graph $\Delta(G)$. Under these conditions, it follows from a result of Pálfy that $\Delta(G)$ consists of two connected components. Another result of Pálfy's gives an inequality relating the sizes of these two connected components. In this paper, we calculate the number of possible component size pairs that satisfy Pálfy's inequality. Additionally, for a fixed positive integer $n$, the number of distinct graph orders for which exactly $n$ component size pairs satisfy Pálfy's inequality is shown.\vskip 5mm

\vskip 5mm


\end{abstract}

\section{Introduction}\label{s:1}

Throughout this paper, $G$ is a finite solvable group. Let Irr($G$) denote the set of irreducible ordinary characters of $G$ and let cd($G$)$=\{\chi(1):\chi\in$ Irr($G$)$\}$, be the set of irreducible character degrees of $G$. We write $\rho(G)$ for the set of primes dividing the elements of cd($G$). The character degree graph of $G$, denoted $\Delta(G)$, is defined as the graph having $\rho(G)$ as its vertex set; there is an edge between primes $p,q\in\rho(G)$ if $pq$ divides $\chi(1)\in$ cd($G$) for some character $\chi\in$ Irr($G$). Note that each $p\in\rho(G)$ corresponds to a vertex $v$ of $\Delta(G)$, and it is commonplace to refer to them interchangeably. Recall also that the order of a graph is the cardinality of its vertex set. Character degree graphs have been studied in many places (for a summary see \cite{MR2397031} and for background see \cite{MR0460423} and Chapter 5, Section 18 of \cite{0521397391} ).\\

In the main theorem of \cite{MR1684505}, Pálfy shows that if $G$ is a solvable group, then for any three vertices of $\rho(G)$, at least two must be adjacent in $\Delta(G)$. Equivalently, the complement graph $\overline{\Delta}(G)$ contains no 3-cycles. Graphs which satisfy this result are commonly said to satisfy Pálfy's condition, and doing so imposes a great deal of structure on the character degree graphs of solvable groups. For example, when $G$ is solvable and $\Delta(G)$ is disconnected, then by Pálfy's condition, $\Delta(G)$ must have exactly two connected components, each of which is a complete graph. Because of this, we see that each disconnected $\Delta(G)$ is completely determined by the sizes of its two connected components. Throughout this paper, we will often refer to pairs of connected component sizes as a stand-in for the associated disconnected character degree graph.\\

In Theorem 3 of \cite{MR1877790}, Pálfy proves another result regarding the sizes of these two complete connected components. Suppose that the two connected components have sizes $a$ and $b$, and without loss of generality, suppose $a\leq b$, then it was shown that $b\geq 2^a-1$. We will refer to this result as Pálfy's inequality, and it motivates our main result.\\

\begin{theorem*}
    Let $\alpha$ be any positive integer. Then there are exactly $2^{\alpha}+1$ possible graph orders for which exactly $\alpha$ pairs of connected component sizes satisfy Pálfy's inequality.
\end{theorem*}\vspace{.20cm}

As of the time of writing, it is still unknown whether all possible pairs of component sizes actually occur within $\Delta(G)$ for some finite solvable group $G$. Note that while the results of this paper are motivated by the study of irreducible character degrees of finite solvable groups, they are actually more combinatorial and number-theoretic in nature.\\

We begin Section 2 with a rough count of the number of possible pairs of connected component sizes for a graph of fixed order $n$. Afterwards, we define a function, in terms of the order of $\Delta(G)$, which counts the number of pairs of component sizes that satisfy Pálfy's inequality, and prove that this count is correct. Finally, we prove the main result. In Section 3, calculations are given to illustrate the surprisingly small number of possible disconnected graphs which can occur as the order of the character degree graph gets large.\\

We note that the work in this paper was completed by the second author as a Ph.D. candidate under the supervision of the first author at Kent State University. The contents of this paper may appear as part of the second author’s Ph.D. dissertation.

\section{Results}\label{s:2}

We begin this section with a simple result that holds for more general graphs.

\begin{lemma}
    Let $\Gamma$ be a graph of order $n\geq 2$. Suppose $\Gamma$ is disconnected with two connected components. Then the number of distinct possible pairs of component sizes is $\lfloor n/2\rfloor$.
\end{lemma}

\begin{proof} If $n$ is odd, then $n=2m+1$ for some positive integer $m$. Let $(a,b)$ represent a possible partition of $n$ into two connected components of sizes $a$ and $b$. We then have that the total number of such partitions is given by the set $\mathcal{C}=\{(a,2m+1-a):a\in\{1,\cdots,m\}\}$. Thus we have $|\mathcal{C}|=m=\lfloor m+\frac{1}{2}\rfloor=\lfloor (2m+1)/2\rfloor=\lfloor n/2\rfloor$.\\

If $n$ is even, then $n=2m$ for some positive integer $m$, and a similar counting argument yields $\mathcal{C}=\{(a,2m-a):a\in\{1,\cdots,m\}\}$. It follows that $|\mathcal{C}|=m=\lfloor 2m/2\rfloor=\lfloor n/2\rfloor$, as desired.

\end{proof}

We now look to refine the previous result in the context of character degree graphs of finite solvable groups. For a given order $n$, we define the function $c(n)$ (for ``connected components") to denote the number of distinct ways in which a graph of order $n$ can be partitioned into connected component sizes which satisfy Pálfy's inequality. Namely, we have the following result:\vspace{.15in}

\begin{lemma}
    Let $n\geq 2$ be a positive integer. Then $$c(n)=\max\{\alpha\in\mathbb{Z}^{+}: n\geq 2^{\alpha}+\alpha+1\}$$ yields the number of ways a graph of order $n$ can be partitioned into two connected components whose sizes satisfy Pálfy's inequality.
\end{lemma}

\begin{proof} Let $\Gamma$ be a graph of order $n$ and let $\alpha$ and $\beta$ denote the sizes of the two connected components of $\Gamma$. Without loss of generality, suppose $\alpha\leq\beta$. To satisfy Pálfy's inequality, we must have $\beta\geq 2^{\alpha}-1$. Noting that $n=\alpha+\beta$, it follows that we must have $n\geq 2^{\alpha}+\alpha-1$. It is not hard to see that the right side of this inequality is strictly increasing in $\alpha$, and so if $\alpha=a$ satisfies the inequality, we have that $a-1, a-2, \cdots, 1$ are also solutions. Thus, for a given order $n$, the total number of pairs of component sizes $(\alpha,\beta)$ which satisfy Pálfy's inequality is given by $\max\{\alpha\in\mathbb{Z}^+:n\geq 2^{\alpha}+\alpha-1\}$. Since this is exactly the definition of $c(n)$, the proof is complete.
    
\end{proof}

With the notation established by Lemma 2 in mind, it follows that for a disconnected character degree graph of order $n$, once $c(n)$ is known, the possible component size pairs are $(n-1,1), (n-2,2), \cdots, (n-c(n),c(n))$. We now state the main result in a slightly different form:\vspace{.15in}

\begin{theorem*}
    Let $\alpha$ be any positive integer. There are exactly $2^{\alpha}+1$ positive integers $n\geq 2$ which satisfy $c(n)=\alpha$.
\end{theorem*}

\begin{proof}
    By definition, $c(n)$ is a non-decreasing function of $n$. It follows that the minimal $n$ such that $c(n)=\alpha$ occurs when equality is achieved in the definition; that is, when $n=2^{\alpha}+\alpha+1$.\\

    Let $n_1,n_2\in\mathbb{Z}^+$ be minimal such that $c(n_1)=\alpha$ and $c(n_2)=\alpha+1$. It follows that:

    \begin{align*}
        n_2-n_1 &= (2^{\alpha+1}+(\alpha+1)-1)-(2^{\alpha}+\alpha-1) \\
        &= 2^{\alpha+1}-2^{\alpha}+1 \\
        &= 2^{\alpha}(2-1)+1 \\
        &= 2^{\alpha}+1.
    \end{align*}

    Since $n_1,n_2$ were chosen minimally, it follows that there are exactly $2^{\alpha}+1$ positive integers satisfying $c(n)=\alpha$.
    
\end{proof}

\section{Calculations}\label{s:3}
In this section, we provide calculations to better illustrate the consequences of our results. We say that a graph $\Gamma$ ``occurs" as a character degree graph if $\Gamma=\Delta(G)$ for some group $G$. A common question regarding character degree graphs concerns which graphs of a fixed order can occur as $\Delta(G)$ (for example, see \cite{MR2125077} or \cite{MR3988190}). How many possible disconnected graphs of order $n$ could occur as $\Delta(G)$ when $n$ gets large? Calculating $c(n)$ shows that the number of possibilities is surprisingly few, as illustrated in Table 1:

    \bgroup
    \def\arraystretch{1.5}
    \begin{table}[htb]
    \caption*{Table 1}
    \noindent\begin{minipage}[c]{0.5\textwidth}
    \centering
    \begin{tabular}{c|c}
        $n$ & $c(n)$ \\\hline
        $10$ & 3\\\hline
        $100$ & 6\\\hline
        $1,000$ & 9\\\hline
        $10,000$ & 13\\\hline
        $100,000$ & 16\\\hline
        $1,000,000$ & 19
    \end{tabular}
    \end{minipage}
    \begin{minipage}[c]{0.5\textwidth}
    \centering
    \begin{tabular}{c|c}
        $n$ & $c(n)$ \\\hline
         $10^9$ & 29\\\hline
        $10^{10}$ & 33\\\hline
        $10^{15}$ & 49\\\hline
        $10^{20}$ & 66\\\hline
        $10^{25}$ & 83\\\hline
        $10^{30}$ & 99 
    \end{tabular}
    \end{minipage}
    \end{table}
    \egroup

\FloatBarrier Note that when considering $10^{30}$ vertices, Lemma 1 provides $5\times 10^{29}$ possible pairs of component sizes, yet, computing $c(n)$ reveals that the number of possible component size pairs that can occur as $\Delta(G)$ is still fewer than one hundred.\\

Additionally, for a fixed number of possible disconnected character degree graphs, our main result yields the possible orders of graph that must be considered, as shown in Table 2:

    \bgroup
    \def\arraystretch{1.5}
    \begin{table}[htb]
    \caption*{Table 2}
    \centering
    \begin{tabular}{c|c}
        Number of possibilities for disconnected $\Delta(G)$ & Possible graph orders \\\hline
        $1$ & $2-4$\\\hline
        $2$ & $5-9$\\\hline
        $3$ & $10-18$\\\hline
        $4$ & $19-35$\\\hline
        $5$ & $36-68$\\\hline
        $6$ & $69-133$\\\hline
        $7$ & $134-262$\\\hline
        $8$ & $263-519$\\\hline
        $9$ & $520-1,032$\\\hline
        $10$ & $1,033-2,057$
    \end{tabular}
    \end{table}
    \egroup

For example, when considering which graphs of order between 36 and 68 can occur as $\Delta(G)$, there will be exactly 5 disconnected possibilities to check. Moreover, one would not have to check more than 10 possible disconnected graphs unless the character degree graphs being considered had at least 2,058 vertices.\\

\printbibliography

\end{document}